\documentclass[12pt]{article}
\usepackage{makecell}
\usepackage{mathrsfs}
\usepackage{mathtools} 
\usepackage{amsmath,amsfonts,amssymb,rotating,amsthm}
\usepackage{diagbox} 
\usepackage{tikz} 
\usepackage{tabularx} 
\usepackage{float} 
\allowdisplaybreaks  
\usepackage{longtable} 
\usepackage{indentfirst} 
\usepackage{hyperref}
\usepackage{array}
\usepackage{setspace}
\usepackage{cleveref}
\def\qed{\nopagebreak\hfill{\rule{4pt}{7pt}}}
\newtheorem{thm}{Theorem}[section]

\newtheorem{lem}[thm]{Lemma}

\newtheorem{conj}[thm]{Conjecture}
\theoremstyle{remark}

\numberwithin{equation}{section}

\begin{document}
\begin{center}
{\Large \bf  Laguerre inequalities and determinantal inequalities for the finite difference of the partition functions}
\end{center}

\begin{center}
Eve Y.Y. Yang\\[8pt]
Department of Mathematics\\
Tianjin University\\
Tianjin 300072, P. R. China\\[6pt]
Email: {\tt yangyaoyao@tju.edu.cn\tt }
\end{center}

\vspace{0.3cm} \noindent{\bf Abstract.} 
The paper aims to establish the Tur\'an inequalities, the Laguerre inequalities (order $2$), and the determinantal inequalities (order $3$) for $\Delta p(n)$ and $\Delta \bar{p}(n)$, where $\Delta f(n)$ is the first-order forward difference of a sequence $f(n)$. The functions $p(n)$ and $\bar{p}(n)$ denote the partition function and overpartition function, respectively. Conjectures for thresholds of Laguerre inequalities (order $m$) and positivity of $m$-order determinants are proposed, extending to 
$\Delta^k p(n)$ and $\Delta^k \bar{p}(n)$, with $1 \leq m \leq 11$ and $1 \leq k \leq 5$.


\section{Introduction}
A positive integer partition of $n$ is a non-increasing sequence of positive integers $\lambda=(\lambda_1,\ldots,\lambda_\ell)$ with $\sum_{i=1}^{\ell}\lambda_i=n$. An overpartition of $n$ allows overlining the first occurrence of each distinct part in the partition. The number of partitions and overpartitions of $n$ are denoted by $p(n)$ and $\bar{p}(n)$, respectively.

Various inequalities concerning $p(n)$ and $\bar{p}(n)$ have been explored using analytical tools, incorporating the Hardy-Ramanujan-Rademacher formulae for $p(n)$ and $\bar{p}(n)$ 
and associated error bounds detailed in \cite{Engel, Lehmer1,Lehmer2, Liu, Zuckerman}.

 The first-order forward difference of a sequence $f(n)$ is defined as 
\[\Delta f(n)=f(n+1)-f(n).\]

Iterating this operation yields the $k$-th order forward difference, 
\begin{equation*}
	\Delta^{k}f(n)=\Delta(\Delta^{k-1}f(n)).
	\end{equation*}

In 1977, Good \cite{Good} conjectured $\Delta^{r}p(n)\geq 0$ for a positive integer $n(r)$. 
Employing the Hardy-Ramanujan-Rademacher formula for $p(n)$, Gupta \cite{Gupta} proved $\Delta^{r}p(n)>0$ for sufficiently large $n$. 
 Subsequently, Odlyzko \cite{Odlyzko} confirmed Good's conjecture and derived the asymptotic formula for  $n(r)$: \[n(r)\sim (6/\pi^2)r^2\log^2 r \quad \text{as} \quad r\rightarrow \infty.\]
Knessl and Keller \cite{Knessl0,Knessl} refined this asymptotic approximation $n'(r)$ of $n(r)$ by deriving a recursion equation for $\Delta^{r}p(n)$ with respect to $n$ and solving it asymptotically. They showed that \[|n'(r)-n(r)|\leq2, \quad \text{for all}\quad r\leq75.\] 
  Almkvist \cite{Almkvist,Almkvist2} discovered a more explicit formula for $\Delta^{r}p(n)>0$ and established specific equations satisfied by $n(r)$.
  Wang, Xie and Zhang \cite{wang1} extended this to $\Delta^{r}\bar{p}(n)>0$ and provided an upper bound for $\Delta^{r}\bar{p}(n)$ for any $r\geq 1$. 
Yang \cite{Yang} extended this approach effectively to $\Delta^{r}q(n)$ and $\{\Delta^{r}\triangle_{k}(n)\}_ {k=1,2}$, where $q(n)$ is the distinct partition function and $\triangle_{k}(n)$ is the broken $k$-diamond partition function.

Consider the $j$-shifted differences of a sequence $f(n)$, denoted as $\Delta_{j}f(n)=f(n)-f(n-j)$, where $1\leq j<n$. Notably, for $j=1$, this corresponds to the first-order backward difference, distinct from the forward difference defined earlier. Gomez, Males and Rolen \cite{Gomez} proved that
\[\Delta^{2}_{j}p(n)=\Delta_{j}\left(\Delta_{j}p(n)\right)=p(n)-2p(n-j)+p(n-2j)>0\] 
by using the Hardy-Ramanujan-Rademacher formula for $p(n)$ and Lehmer's error bound. They conjectured $\Delta^{r}_{j}p(n)\geq0$ for a fixed $j$, any given $r$, and sufficiently large $n$. The case for $j=1$ had been previously proved by Gupta \cite{Gupta}.
Banerjee \cite{Banerjee} proved $\Delta^{2}_{j}p(n)\geq0$ and $\Delta^{2}_{j}\bar{p}(n)>0$ using an elementary combinatorial approach. 
Numerous related works exist \cite{Canfield,Chen-Wang-Xie,Knessl00,Merca}.

The Tur\'an-type inequalities arise in the study of the Maclaurin coefficients of real entire functions in the Laguerre-P\'olya class. For a more detailed study, we refer to \cite{Craven,Dimitrov,Levin,Polya,Rahman}.

A sequence $\{a_n\}_{n\geq0}$ of real numbers is log-concave if it satisfies the (second order) Tur\'an inequality $a_n^2\geq a_{n-1}a_{n+1}$ for $n\geq1$. It satisfies the third-order Tur\'an inequality if, for $n\geq1$,
\begin{equation*}
  4(a_n^2-a_{n-1}a_{n+1})(a_{n+1}^2-a_na_{n+2})\geq(a_na_{n+1}-a_{n-1}a_{n+2})^2.
\end{equation*}
Note that the log-concavity of a sequence $f(n)$ implies that $-\Delta^2 \log  f(n)\geq 0$.

Define the operator $\mathcal{L}$ as
\begin{eqnarray*}
\mathcal{L}a_{n}=a_{n}^2-a_{n-1}a_{n+1}\quad~~\text{and}\quad~~\mathcal{L}^ra_{n}=\mathcal{L}(\mathcal{L}^{r-1}a_{n}).
\end{eqnarray*}
A sequence $\{a_n\}_{n\geq0}$ is said to be $r$-log-concave if, for $1\leq k\leq r$ and $n\geq k$, $\{\mathcal{L}^ka_{n}\}_{n\geq1}$ are all nonnegative sequences.
It should be noted that $2$-log-concavity is consistent with the double Tur\'an inequality.

The Tur\'an-type inequalities are closely related to the Jensen polynomials. The Jensen polynomials of degree $d$ and shift $n$ associated with an arbitrary real sequence $\{a_n\}_{n\geq0}$ are defined by
\begin{equation*}
J^{d,n}_a(X)\coloneqq\sum_{j=0}^d{d\choose j}a_{n+j}X^j.
\end{equation*}

In general, a sequence $\{a_{n}\}_{n\geq 0}$ satisfies the Tur\'an inequality of order $d$ at $n$ if and only if $J^{d,n-1}_a(X)$ is hyperbolic. A real polynomial is said to be hyperbolic if all of its zeros are real.

P\'olya \cite{Polya} proved that the Riemann hypothesis (RH) is equivalent to the hyperbolicity of the polynomials $J^{d,n}_{\hat{\gamma}_k}$ for all $d,n\geq0$, where 
 $\{\hat{\gamma}_k\}_{k\geq0}$ is the sequence of Maclaurin coefficients of the Riemann Xi-function. Due to the difficulty of proving RH, research has  
concentrated on proving hyperbolicity for all $n\geq0$ when $d$ is small. The cases $d=2$ and $d=3$ imply that $\{\hat{\gamma}_k\}_{k\geq0}$ satisfies both the Tur\'an inequality and the third-order Tur\'an inequality, as proved by Csordas, Norfolk and Varga \cite{Csordas}, and Dimitrov and Lucas \cite{Dimitrov2}, respectively.

 Further properties of the Tur\'an-type inequalities, the Laguerre-P\'olya class, the Jensen polynomials, and the Riemann Xi-function are detailed in
\cite{Craven,Csordas,Csordas2,Dimitrov,Dimitrov2,Jensen,Polya}.

The Laguerre inequality, also relevant to the Laguerre-P\'olya class, initially formulated for a polynomial function $f(x)$, is expressed as
\begin{equation}\label{Laguerre inequality}
L_1(f(x))\coloneqq{f'(x)}^2-f(x)f''(x)\geq 0,\quad x\in\mathbb{R}.
\end{equation} 
 
Laguerre \cite{Laguerre} stated that if $f(x)$ is hyperbolic, then it satisfies \eqref{Laguerre inequality}. Skovgaard \cite{Skovgaard} asserted that if $f(x)$ belongs to the Laguerre-P\'olya class, then for any $p\in\mathbb{Z}_{\geq0}$, $f^{p}(x)$ satisfies \eqref{Laguerre inequality}, where $f^{p}(x)$ represents the $p$-th derivative of $f(x)=f^{0}(x)$.

In 1913, Jensen \cite{Jensen} 
defined a higher order generalization of the Laguerre inequality, namely, for $f(x)$ in the Laguerre-P\'olya class,
\begin{equation*}\label{L_{m}}
L_{n}(f(x))\coloneqq\frac{1}{2}\sum_{k=0}^{2n}(-1)^{n+k}{{2n}\choose k}f^{k}(x)f^{2n-k}(x) \geq 0.
\end{equation*}
These inequalities, often referred to as Laguerre-type inequalities, yield the classical Laguerre inequality \eqref{Laguerre inequality} for $n=1$. Csordas and Varga \cite{Csordas2} showed that if a function $f(x)$ satisfies $ L_{n}(f(x))\geq0 $ for all $n$ and all $x\in \mathbb{R}$, then $f(x)$ belongs to the Laguerre-P\'olya class. 

 Recently, Wang and Yang \cite{wang2} considered whether the discrete sequence $\{a_n\}_{n\geq0}$ has the similar results with higher order Laguerre inequalities. They defined that a sequence $\{a_n\}_{n\geq0}$ satisfies the Laguerre inequality of order $m$ if
\begin{equation*}
L_{m}(a_n)\coloneqq \frac{1}{2}\sum_{k=0}^{2m}(-1)^{k+m}{{2m}\choose k}a_{n+k}a_{2m-k+n} \geq 0.
\end{equation*}
 
For $m = 1$, 
this reduces to
\begin{equation*}
a_{n+1}^2-a_{n}a_{n+2}\geq 0,
\end{equation*}
which corresponds to the log-concavity of $\{a_n\}_{n\geq 0}$.

Wang and Yang \cite{wang2,wang4} established the Laguerre inequality of order $2$ for some discrete sequences, including $p(n)$, $\bar{p}(n)$, $q(n)$, and others. Wagner \cite{Wagner} proved that $p(n)$ satisfies the Laguerre inequalities of any order as $n\rightarrow \infty$ and 
conjectured the thresholds of the $m$-rd Laguerre inequalities of $p(n)$ for $m\leq 10$. Dou and Wang \cite{dw} proved Wagner's conjecture for $3\leq m\leq 9$. 
For more  work on the relation between the Laguerre inequality, the Tur\'an inequality and the Laguerre-P\'olya class, refer to \cite{Cardon, Craven, Csordas5,Csordas3, K.Dilcher, W.H.Foster,Patrick1, Patrick}.

Chen \cite{Chen3} conducted a comprehensive study on inequalities related to invariants of a binary form. He examined the following three invariants associated with the hyperbolicity of the Jensen polynomial of degree $4$.
\begin{equation*}
\begin{split}
A(a_0,a_1,a_2,a_3,a_4)&=a_0a_4-4a_1a_3+3a_2^2,\\[9pt]
B(a_0,a_1,a_2,a_3,a_4)&=-a_0a_2a_4+a_2^3+a_0a_3^2+a_1^2a_4-2a_1a_2a_3,\\[9pt]
I(a_0,a_1,a_2,a_3,a_4)&=A(a_0,a_1,a_2,a_3,a_4)^3-27B(a_0,a_1,a_2,a_3,a_4)^2.
\end{split}
\end{equation*}

Throughout the rest of this paper, we shall use the notations $A$, $B$, and $I$ to denote $A(a_0,a_1,a_2,a_3,a_4)$, $B(a_0,a_1,a_2,a_3,a_4)$, and $I(a_0,a_1,a_2,a_3,a_4)$
respectively, associated with a sequence $\{a_n\}_{n\geq0}$.

Indeed, $A>0$ coincides with the Laguerre inequality of order $2$. $B>0$ can be expressed as
\[
\left|\begin{array}{ccc}
		a_2 & a_3 &  a_4\\
		a_1 & a_2 & a_3\\
		a_0 & a_1 & a_2\\
		\end{array} \right|>0,
\]
which has been proven to be equivalent to $2$-log-concavity in combinatorics. For $p(n)$, $B>0$ has been showed by Hou and Zhang \cite{hz}, and Jia and Wang \cite{Jia-Wang-2018} independently. Recently, Wang and Yang \cite{wang3} proved 
 $\det (a_{n-i+j})_{1\leq i,j\leq 4}>0$ for the sequence $\{a_n\}_{n\geq0}$ involving $p(n)$ and $\bar{p}(n)$ and provided an iterated approach to compute $\det (a_{n-i+j})_{1\leq i,j\leq m}$ for any positive integer $m$.

Recent work by various mathematicians has established Tur\'an inequalities of order $d$ for $p(n)$. Nicolas \cite{Nicolas} and DeSalvo and Pak \cite{Desalvo} independently confirmed the log-concavity of $\{p(n)\}_{n\geq 25}$. 
Chen \cite{Chen3} conjectured and Chen, Jia and Wang \cite{Chen-Jia-Wang-2017} proved the third-order Tur\'an inequality for $\{p(n)\}_{n\geq 95}$. 
 Chen, Jia and Wang \cite{Chen-Jia-Wang-2017} extended this to propose a conjecture for Tur\'an inequalities of order $d$ for $p(n)$ with $d\geq 4$ and sufficiently large $n$, which was proven by Griffin, Ono, Rolen and Zagier \cite{Zagier}. Larson and Wagner \cite{larson} provided the thresholds for Tur\'an inequalities of order $4$ and $5$.

 Tur\'an inequalities for other partition functions have been extensively investigated, refer to \cite{Craig,DJ,Dong,Jia,wang4,Yang1} for further details.
 
 Let 
 \begin{equation*}
 \Delta p(n)=p(n+1)-p(n)\quad \text{and} \quad\Delta \bar{p}(n)=\bar{p}(n+1)-\bar{p}(n).
 \end{equation*}
  This paper establishes the Tur\'an inequalities, the Laguerre inequalities of order $2$, and the determinantal inequalities of order $3$ for $\Delta p(n)$ and $\Delta \bar{p}(n)$. Our main tool, due to Dou and Wang \cite{dw}, relies on Lemma \ref{boundp} for $p(n)$ and Lemma \ref{boundbarp} for $\bar{p}(n)$, both provided by Wang and Yang \cite{wang3}.
\begin{lem}[{\cite[Lemma 2.1]{wang3}}]\label{boundp}
For any given integer $l$, there exists $N(l)$ such that for all $n \geq N(l)$,
\begin{equation*}
			\frac{\sqrt{12} \pi^{2} e^{\mu}}{36 \mu^{2}}\left(1-\frac{1}{\mu}-\frac{1}{\mu^{l}}\right)<p(n)<\frac{\sqrt{12} \pi^{2} e^{\mu}}{36 \mu^{2}}\left(1-\frac{1}{\mu}+\frac{1}{\mu^{l}}\right),
	\end{equation*}
where $\mu$ is the abbreviation for $\mu(n)$, and 
\begin{equation}\label{mu}
  \mu(n)=\frac{\pi \sqrt{24 n-1}}{6}.
\end{equation}
\end{lem}

\begin{lem}[{\cite[Lemma 5.1]{wang3}}]\label{boundbarp}
For any given integer $t$, there exists $N(t)$ such that for all $n \geq N(t)$,
\begin{equation*}
			\frac{\pi^{2} e^{\bar{\mu}}}{8 \bar{\mu}^{2}}\left(1-\frac{1}{\bar{\mu}}-\frac{1}{\bar{\mu}^{t}}\right)<\bar{p}(n)<\frac{\pi^{2} e^{\bar{\mu}}}{8 \bar{\mu}^{2}}\left(1-\frac{1}{\bar{\mu}}+\frac{1}{\bar{\mu}^{t}}\right),
	\end{equation*}
where $\bar{\mu}$ is the abbreviation for $\bar{\mu}(n)$, and 
\begin{equation*}
\bar{\mu}(n)=\pi \sqrt{n}.
\end{equation*}
\end{lem}
Wang and Yang \cite{wang3} inferred that $n \geq N(k)$ ($k=l,t$) in the above Lemmas are solutions to $6 e^{-\frac{u}{2}} < \frac{1}{u^k}$ ($u=\mu$, $\bar{\mu}$). 

The remaining of this paper is organized as follows. Section \ref{TIp} establishes the Tur\'an inequalities for $\Delta p(n)$ and $\Delta \bar{p}(n)$. Section \ref{LIp}, following the spirit of Section \ref{TIp}, establishes the Laguerre inequalities for $\Delta p(n)$ and $\Delta \bar{p}(n)$. Similarly, Section \ref{DIp} establishes the determinantal inequalities of order $3$ for $\Delta p(n)$ and $\Delta \bar{p}(n)$. We conclude in Section \ref{Conj} with open problems for further exploration.

\section{The Tur\'an inequalities for $\Delta p(n)$ and $\Delta \bar{p}(n)$}\label{TIp}
In this section, we will establish the Tur\'an inequalities for $\Delta p(n)$ and $\Delta \bar{p}(n)$, detailed in Theorem \ref{TIv} and Theorem \ref{TI_barv}, respectively. 
\begin{thm}\label{TIv}
  For $n \geq 71$, $\Delta p(n)$ satisfies the Tur\'an inequality,
  \begin{equation}\label{TIneq v}
			\left(\Delta p(n)\right)^2>\left(\Delta p(n-1)\right)\left(\Delta p(n+1)\right).
	\end{equation}
\end{thm}
\begin{thm}\label{TI_barv}
 For $n \geq 8$, $\Delta \bar{p}(n)$ satisfies the Tur\'an inequality,
  \begin{equation}\label{TIneq barv}
			\left(\Delta \bar{p}(n)\right)^2>\left(\Delta \bar{p}(n-1)\right)\left(\Delta \bar{p}(n+1)\right).
	\end{equation}
\end{thm}

In fact, the entire approach used to prove Theorem \ref{TIv} is also applicable to Theorem \ref{TI_barv}. Thus, we will provide a detailed proof only for the former.

\begin{proof}[Proof of Theorem \ref{TIv}]
Note that \eqref{TIneq v} can be rewritten as
\begin{equation}\label{rewrite}
  \begin{aligned}
p(n)^2&-p(n-1)p(n+1)-p(n)p(n+1)+p(n+1)^2\\[9pt]
&+p(n-1)p(n+2)-p(n)p(n+2)>0.
\end{aligned}
\end{equation}

To prove it, setting $l=6$ in Lemma \ref{boundp}, obtaining the bounds for $n \geq 391$,
\begin{equation}\label{Boundp}
			e^{\mu} \frac{\sqrt{12} \beta(\mu) \pi^{2}}{36 \mu^{8}}<p(n)<e^{\mu} \frac{\sqrt{12} \alpha(\mu) \pi^{2}}{36 \mu^{8}},
	\end{equation}
where
\begin{equation}\label{Alpha_Beta}
			\alpha(t)=t^{6}-t^{5}+1, \quad ~~\beta(t)=t^{6}-t^{5}-1.
	\end{equation}
Define
\begin{equation}\label{upperlower}
			f(n)\coloneqq e^{\mu} \frac{\sqrt{12} \beta(\mu) \pi^{2}}{36 \mu^{8}},\quad~~
g(n)\coloneqq e^{\mu} \frac{\sqrt{12} \alpha(\mu) \pi^{2}}{36 \mu^{8}}.
	\end{equation}

Then, to establish \eqref{TIneq v}, we aim to show that
\begin{equation*}
  \begin{split}
f(n)^2&-g(n-1)g(n+1)-g(n)g(n+1)+f(n+1)^2\\[9pt]
&+f(n-1)f(n+2)-g(n)g(n+2)>0,
\end{split}
\end{equation*}
which is equivalent to 
\begin{equation}\label{TIneq vp}
  \begin{split}
\frac{f(n)^{2}}{f(n+1)^{2}}&-\frac{g(n-1)g(n+1)}{f(n+1)^{2}}-\frac{g(n)g(n+1)}{f(n+1)^{2}}+1\\[9pt]
&+\frac{f(n-1)f(n+2)}{f(n+1)^{2}}-\frac{g(n)g(n+2)}{f(n+1)^{2}}>0.
\end{split}
\end{equation}
For convenience, let
\begin{equation*}
  w=\mu(n-1),\quad x=\mu(n),\quad y=\mu(n+1), \quad z=\mu(n+2).
\end{equation*}
Simplify the left-hand side of the inequality \eqref{TIneq vp} as
\begin{equation}\label{he_vp}
\frac{-h_{1}e^{w-y}-h_{2}e^{x-y}-h_{3}e^{x-2 y+z}+h_{4}e^{2 x-2 y}+h_{5}e^{w-2 y+z}+h_{6}}{h_{6}},
\end{equation}
where
\begin{equation}\label{h1_6}
		\begin{aligned}
			&h_{1}=x^{16} y^8 z^8 \alpha(w)\alpha(y),\quad
			&&h_{2}=w^8 x^8 y^8 z^8 \alpha(x)\alpha(y),\\
			&h_{3}=w^8 x^8 y^{16} \alpha(x)\alpha(z),\quad
			&&h_{4}=w^8 y^{16} z^8 \beta(x)^2,\\
			&h_{5}=x^{16} y^{16} \beta(w)\beta(z),\quad
			&&h_{6}=w^8 x^{16} z^8 \beta(y)^2.
		\end{aligned}
	\end{equation}

Clearly, $h_{6}>0$ for $n\geq 1$. Thus, our focus is only on proving the positivity of the numerator in \eqref{he_vp}, i.e.,
\begin{equation}\label{Numerator}
-h_{1}e^{w-y}-h_{2}e^{x-y}-h_{3}e^{x-2 y+z}+h_{4}e^{2 x-2 y}+h_{5}e^{w-2 y+z}+h_{6}>0.
\end{equation}

To achieve this, we need to estimate $h_{1}$, $h_{2}$, $h_{3}$, $h_{4}$, $h_{5}$, $h_{6}$, $e^{w-y}$, $e^{x-y}$, $e^{x-2 y+z}$, $e^{2 x-2 y}$ and $e^{w-2 y+z}$. We prefer to use the following equations to estimate $w$, $x$, and $z$.
For $n\geq 2$,
\begin{equation}\label{wxz}
w=\sqrt{y^{2}-\frac{4 \pi^{2}}{3}}, \quad~x=\sqrt{y^{2}-\frac{2 \pi^{2}}{3}}, \quad~z=\sqrt{y^{2}+\frac{2 \pi^{2}}{3}}.
\end{equation}
Their Taylor expansions are
\begin{align*}
    w&=y-\frac{2 \pi^{2}}{3 y}-\frac{2 \pi^{4}}{9 y^{3}}-\frac{4 \pi^{6}}{27 y^{5}}+O\left(\frac{1}{y^{7}}\right),\notag\\[9pt]
    x&=y-\frac{\pi^{2}}{3 y}-\frac{\pi^{4}}{18 y^{3}}-\frac{\pi^{6}}{54 y^{5}}+O\left(\frac{1}{y^{7}}\right),\\[9pt]
    z&=y+\frac{\pi^{2}}{3 y}-\frac{\pi^{4}}{18 y^{3}}+\frac{\pi^{6}}{54 y^{5}}+O\left(\frac{1}{y^{7}}\right).\notag
\end{align*}
Then, it can be checked that for $n \geq 2$,
\begin{equation}\label{Ineq_wxz}
	w_1<w<w_2,\quad~~~~~~x_1<x<x_2,\quad~~~~~~z_1<z<z_2,
\end{equation}
where
\begin{equation*}
  \begin{aligned}
    &w_1=y-\frac{2 \pi^{2}}{3 y}-\frac{2 \pi^{4}}{9 y^{3}}-\frac{4 \pi^{6}}{27 y^{5}},\quad
    &&w_2=y-\frac{2 \pi^{2}}{3 y}-\frac{2 \pi^{4}}{9 y^{3}}-\frac{3 \pi^{6}}{27 y^{5}},\\[9pt]
    &x_1=y-\frac{\pi^{2}}{3 y}-\frac{\pi^{4}}{18 y^{3}}-\frac{\pi^{6}}{54 y^{5}},\quad
    &&x_2=y-\frac{\pi^{2}}{3 y}-\frac{\pi^{4}}{18 y^{3}}-\frac{\pi^{6}}{108 y^{5}},\\[9pt]
    &z_1=y+\frac{\pi^{2}}{3 y}-\frac{\pi^{4}}{18 y^{3}},\quad
    &&z_2=y+\frac{\pi^{2}}{3 y}-\frac{\pi^{4}}{18 y^{3}}+\frac{\pi^{6}}{54 y^{5}}.
    \end{aligned}
\end{equation*}

Next, we 
estimate $h_{1}$, $h_{2}$, $h_{3}$, $h_{4}$ and $h_{5}$.
Applying \eqref{Ineq_wxz} to the definitions \eqref{Alpha_Beta} of $\alpha(t)$ and $\beta(t)$, we obtain that for $n\geq 2$,
\begin{equation*}
\begin{split}
   &w^{6}-w_{2} w^{4}+1 < \alpha(w) < w^{6}-w_{1} w^{4}+1, \\[9pt]
   &x^{6}-x_{2} x^{4}+1 < \alpha(x) < x^{6}-x_{1} x^{4}+1, \\[9pt]
   &z^{6}-z_{2} z^{4}+1 < \alpha(z) < z^{6}-z_{1} z^{4}+1, \\[9pt]
   &w^{6}-w_{2} w^{4}-1 < \beta(w) < w^{6}-w_{1} w^{4}-1, \\[9pt]
   &x^{6}-x_{2} x^{4}-1 < \beta(x) < x^{6}-x_{1} x^{4}-1, \\[9pt]
   &z^{6}-z_{2} z^{4}-1 < \beta(z) < z^{6}-z_{1} z^{4}-1.
\end{split}
\end{equation*}
Substituting these $\alpha(t)$ and $\beta(t)$ into $h_{1}$, $h_{2}$, $h_{3}$, $h_{4}$, $h_{5}$, we get 
\begin{equation}\label{h_ibound}
\begin{split}
&h_{1}<x^{16} y^8 z^8 \left(w^{6}-w_{1} w^{4}+1\right)\left(y^{6}-y^{5}+1\right),\\[9pt]
&h_{2}<w^8 x^8 y^8 z^8 \left(x^{6}-x_{1} x^{4}+1\right)\left(y^{6}-y^{5}+1\right),\\[9pt]
&h_{3}<w^8 x^8 y^{16} \left(x^{6}-x_{1} x^{4}+1\right)\left(z^{6}-z_{1} z^{4}+1\right),\\[9pt]
&h_{4}>w^8 y^{16} z^8 \left(x^{6}-x_{2} x^{4}-1\right)^2,\\[9pt]
&h_{5}>x^{16} y^{16} \left(w^{6}-w_{2} w^{4}-1\right)\left(z^{6}-z_{2} z^{4}-1\right).
\end{split}
\end{equation}

Next we turn to estimate $e^{w-y}$, $e^{x-y}$, $e^{x-2 y+z}$, $e^{2 x-2 y}$ and $e^{w-2 y+z}$.
 By \eqref{Ineq_wxz}, it can be seen that for $n\geq 2$,
\begin{align*}
w_1-y&<w-y<w_2-y,\\[9pt]
x_1-y&<x-y<x_2-y,\\[9pt]
x_1-2y+z_{1}&<x-2 y+z<x_2-2y+z_{2},\\[9pt]
w_1-2y+z_{1}&<w-2 y+z<w_2-2y+z_{2}.
\end{align*}
This implies that
\begin{equation}\label{eIneq}
\begin{split}
&e^{w_1-y}<e^{w-y}<e^{w_2-y},\\[9pt]
&e^{x_1-y}<e^{x-y}<e^{x_2-y},\\[9pt]
&e^{2x_1-2y}<e^{2x-2y}<e^{2x_2-2y},\\[9pt]
&e^{x_1-2y+z_{1}}<e^{x-2 y+z}<e^{x_2-2y+z_{2}},\\[9pt]
&e^{w_1-2y+z_{1}}<e^{w-2 y+z}<e^{w_2-2y+z_{2}}.
\end{split}
\end{equation}

To provide a feasible bound for $e^{x-2 y+z}$ and $e^{w-2 y+z}$, we define
\begin{equation*}
\begin{split}
\Phi(t)&=1+t +\frac{t ^2}{2}+\frac{t ^3}{6}+\frac{t ^4}{24}+\frac{t ^5}{120},\\[9pt]
\phi(t)&=1+t +\frac{t ^2}{2}+\frac{t ^3}{6}+\frac{t ^4}{24}+\frac{t ^5}{120}+\frac{t ^6}{720}.
\end{split}
\end{equation*}

For $t<0$, it is evident that
\begin{equation}\label{Phi-phit}
\phi(t)<e^{t}<\Phi(t).
\end{equation}

 To apply this to \eqref{eIneq}, it suffices to show the negativity of $w_2-y$, $x_2-y$, $2x_2-2y$, $x_2-2y+z_{2}$ and $w_2-2y+z_{2}$. A direct calculation reveals that
\begin{equation}\label{e^bound}
\begin{aligned}
&w_2-y=-\frac{\pi^6+2\pi^4y^2+6\pi^2y^4}{9y^5},\\[9pt]
&x_2-y=-\frac{\pi^6+6\pi^4y^2+36\pi^2y^4}{108y^5},\\[9pt]
&x_2-2y+z_{2}=\frac{\pi^{6}-12 \pi^{4} y^{2}}{108 y^{5}},\\[9pt]
&w_2-2y+z_{2}=-\frac{5 \pi^{6}+15 \pi^{4} y^{2}+18 \pi^{2} y^{4}}{54 y^{5}}.
\end{aligned}
\end{equation}

Clearly, $w_2-y$, $x_2-y$, $2x_2-2y$, $x_2-2y+z_{2}$ and $w_2-2y+z_{2}$ are negative for all $n$. Notably, for their negativity, a judicious choice of the denominator in \eqref{TIneq vp} is essential. Thus, applying \eqref{Phi-phit} to \eqref{eIneq}, we deduce that for $n\geq 2$,
\begin{equation}\label{peP}
\begin{split}
&\phi(w_1-y)<e^{w-y}<\Phi(w_2-y),\\[9pt]
&\phi(x_1-y)<e^{x-y}<\Phi(x_2-y),\\[9pt]
&\phi(2x_1-2y)<e^{2x-2y}<\Phi(2x_2-2y),\\[9pt]
&\phi(x_1-2y+z_{1})<e^{x-2 y+z}<\Phi(x_2-2y+z_{2}),\\[9pt]
&\phi(w_1-2y+z_{1})<e^{w-2 y+z}<\Phi(w_2-2y+z_{2}).
\end{split}
\end{equation}

Now, we are ready for proving \eqref{Numerator}. Combing expressions \eqref{h1_6} and \eqref{wxz}, we can get that $h_{1}$, $h_{2}$, $h_{3}$, $h_{4}$, $h_{5}$, $h_{6}$, $e^{w-y}$, $e^{x-y}$, $e^{x-2 y+z}$, $e^{2 x-2 y}$ and $e^{w-2 y+z}$ are all functions of $y$. For convenience, let
\begin{equation*}
A(y)=-h_{1}e^{w-y}-h_{2}e^{x-y}-h_{3}e^{x-2 y+z}+h_{4}e^{2 x-2 y}+h_{5}e^{w-2 y+z}+h_{6},
\end{equation*}
we need to show $A(y)>0$. Using \eqref{h_ibound} and \eqref{peP}, it holds that for $n\geq 2$,
\begin{align*}
A(y)>&-x^{16} y^8 z^8 \left(w^{6}-w_{1} w^{4}+1\right)\left(y^{6}-y^{5}+1\right)\Phi(w_2-y)\\[9pt]
&-w^8 x^8 y^8 z^8 \left(x^{6}-x_{1} x^{4}+1\right)\left(y^{6}-y^{5}+1\right)\Phi(x_2-y)\\[9pt]
&-w^8 x^8 y^{16} \left(x^{6}-x_{1} x^{4}+1\right)\left(z^{6}-z_{1} z^{4}+1\right)\Phi(x_2-2y+z_{2})\\[9pt]
&+w^8 y^{16} z^8 \left(x^{6}-x_{2} x^{4}-1\right)^2\phi(2x_1-2y)\\[9pt]
&+x^{16} y^{16} \left(w^{6}-w_{2} w^{4}-1\right)\left(z^{6}-z_{2} z^{4}-1\right)\phi(w_1-2y+z_{1})\\[9pt]
&+w^8 x^{16} z^8 \beta(y)^2.
\end{align*}

Represent the right-hand side of the above inequality as $A_1(y)$. 
With Mathematica, we can quickly simplify $A_1(y)$ as 
\begin{equation*}
A_1(y)={\sum_{k=0}^{66} a_k y^{k}\over 2^{18}3^{37} 5 y^{27}}.
\end{equation*}

Note that all the coefficients of $A_1(y)$ can be known, and the values of $a_{64}$, $a_{65}$, $a_{66}$ are given below,
\begin{equation*}
 \begin{aligned}
a_{64}&=2^{15} 3^{29} 5 \left(629856+209952 \pi^2-7776 \pi^6+62856 \pi^8-2133 \pi^{10}+55 \pi^{12}\right),\\[9pt]
a_{65}&=-2^{14} 3^{29} 5 \left(1259712-15552 \pi^6+11988 \pi^8+432 \pi^{10}+13 \pi^{12}\right),\\[9pt]
a_{66}&=2^{18} 3^{33} 5 \pi^6\left(-6 +\pi^2\right).
 	\end{aligned}
\end{equation*}

Thus, for $n \geq 2$, we have
\begin{equation}\label{Ay>}
A(y)>{\sum_{k=0}^{66} a_k y^{k}\over 2^{18}3^{37} 5 y^{27}}.
\end{equation}

Since $y$ is positive for $n\geq 1$, it follows that
\begin{equation*}
\sum_{k=0}^{66} a_k y^{k}>\sum_{k=0}^{64} -|a_k| y^{k}+a_{65} y^{65}+a_{66} y^{66}.
\end{equation*}

For $0\leq k \leq 64$, it can be verified that
\begin{equation*}
-|a_k|y^{k}>-a_{64}y^{64}
\end{equation*}
holds when $y>7.68$.
This leads to
\begin{equation}\label{Ay>>}
\sum_{k=0}^{66} a_k y^{k}>\sum_{k=0}^{64} -|a_k| y^{k}+a_{65} y^{65}+a_{66} y^{66}>(-65a_{64}+a_{65}y+a_{66}y^2)y^{64}.
\end{equation}

Combing \eqref{Ay>} and \eqref{Ay>>}, $A(y)$ is positive provided
\begin{equation*}
-65a_{64}+a_{65}y+a_{66}y^2>0,
\end{equation*}
which holds for $y>126.22$. From the relation between $y$ and $n$ in \eqref{mu} and \eqref{wxz}, we deduce that $y>126.22$ is equivalent to $n\geq 2421$. Thus, \eqref{Numerator} is true for $n\geq 2421$, which implies \eqref{TIneq v}.
Numerical verification further confirms the truth of \eqref{TIneq v} for $71 \leq n \leq 2420$. Consequently, we conclude the proof of Theorem \ref{TIv}.
\end{proof}

To prove Theorem \ref{TI_barv}, we employ Lemma \ref{boundbarp} and adopt notations consistent with the proof of Theorem \ref{TIv}, except adjusting $p(n)$ and $\mu(n)$ to $\bar{p}(n)$ and $\bar{\mu}(n)$, respectively. The proof follows the framework established for Theorem \ref{TIv}. Hence, in the remaining of this section, we will omit some tedious formulae and statements and focus on presenting key modifications resulting from substituting $p(n)$ and $\mu(n)$ with $\bar{p}(n)$ and $\bar{\mu}(n)$, respectively.
 
Considering the substitution of $\bar{p}(n)$ for $p(n)$, we first adjust \eqref{rewrite} to 
\begin{equation*}
  \begin{aligned}
\bar{p}(n)^2&-\bar{p}(n-1)\bar{p}(n+1)-\bar{p}(n)\bar{p}(n+1)+\bar{p}(n+1)^2\\[9pt]
&+\bar{p}(n-1)\bar{p}(n+2)-\bar{p}(n)\bar{p}(n+2)>0.
\end{aligned}
\end{equation*}

 Furthermore, with the substitution of $\mu(n)$ for $\bar{\mu}(n)$, \eqref{Boundp} transforms to
\begin{equation*}
			e^{\bar{\mu}}\frac{\beta(\bar{\mu})\pi^{2}}{8 \bar{\mu}^{8}}<\bar{p}(n)<e^{\bar{\mu}}\frac{\alpha(\bar{\mu})\pi^{2}}{8 \bar{\mu}^{8}},
	\end{equation*}
where
\begin{equation*}
			\alpha(t)=t^{6}-t^{5}+1, \quad ~~\beta(t)=t^{6}-t^{5}-1.
	\end{equation*}
Then \eqref{upperlower} transforms to
\begin{equation*}
			f(n)\coloneqq e^{\bar{\mu}}\frac{\beta(\bar{\mu})\pi^{2}}{8 \bar{\mu}^{8}},\quad~~
g(n)\coloneqq e^{\bar{\mu}}\frac{\alpha(\bar{\mu})\pi^{2}}{8 \bar{\mu}^{8}}.
	\end{equation*}

As $\mu(n)$ changes to $\bar{\mu}(n)$, \eqref{wxz} transforms to, for $n\geq 2$,
\begin{equation*}
w=\sqrt{y^{2}-2\pi^{2}}, \quad~x=\sqrt{y^{2}-\pi^{2}}, \quad~z=\sqrt{y^{2}+\pi^{2}}.
\end{equation*}
Their Taylor expansions are given by
\begin{align*}
    w&=y-\frac{\pi^2}{y}-\frac{\pi^4}{2 y^3}-\frac{\pi^6}{2 y^5}+O\left(\frac{1}{y^{7}}\right),\\[9pt]
    x&=y-\frac{\pi^2}{2 y}-\frac{\pi^4}{8 y^3}-\frac{\pi^6}{16 y^5}+O\left(\frac{1}{y^{7}}\right),\\[9pt]
    z&=y+\frac{\pi^2}{2 y}-\frac{\pi^4}{8 y^3}+\frac{\pi^6}{16 y^5}+O\left(\frac{1}{y^{7}}\right).
\end{align*}
For $n \geq 2$, it can be checked that
\begin{equation*}
	w_1<w<w_2,\quad~~~~~~x_1<x<x_2,\quad~~~~~~z_1<z<z_2,
\end{equation*}
where
\begin{align*}
	w_1&=y-\frac{\pi^2}{y}-\frac{\pi^4}{2 y^3}-\frac{\pi^6}{2 y^5}, & w_2&=y-\frac{\pi^2}{y}-\frac{\pi^4}{2 y^3}-\frac{\pi^6}{4 y^5},\\[9pt]
	x_1&=y-\frac{\pi^2}{2 y}-\frac{\pi^4}{8 y^3}-\frac{\pi^6}{16 y^5}, & x_2&=y-\frac{\pi^2}{2 y}-\frac{\pi^4}{8 y^3}-\frac{\pi^6}{32 y^5},\\[9pt]
	z_1&=y+\frac{\pi^2}{2 y}-\frac{\pi^4}{8 y^3}, & z_2&=y+\frac{\pi^2}{2 y}-\frac{\pi^4}{8 y^3}+\frac{\pi^6}{16 y^5}.
\end{align*}
Furthermore, \eqref{e^bound} should be adjusted to
\begin{align*}
&w_2-y=-\frac{\pi^6+2\pi^4y^2+4\pi^2y^4}{4y^5},\\[9pt]
&x_2-y=-\frac{\pi^6+4\pi^4y^2+16\pi^2y^4}{32y^5},\\[9pt]
&x_2-2y+z_{2}=\frac{\pi^6-8 \pi^4 y^2}{32 y^5},\\[9pt]
&w_2-2y+z_{2}=-\frac{3\pi^6+10\pi^4y^2+8\pi^2y^4}{16y^5}.
\end{align*}

Finally, applying similar arguments as in the proof of Theorem \ref{TIv}, we observe that $A_1(y)$ is a certain polynomial of degree $39$ with a positive leading coefficient. Verification confirms the positivity of $A(y)$ for $n\geq 1641$, establishing the Tur\'an inequality for $\Delta\bar{p}(n)$ in this range. For $8 \leq n \leq 1640$, numerical verification additionally supports the Tur\'an inequality for $\Delta\bar{p}(n)$. Thus, the proof of Theorem \ref{TI_barv} is completed.
\qed

\section{The Laguerre inequalities for $\Delta p(n)$ and $\Delta\bar{p}(n)$}\label{LIp}
In this section, we will consider the Laguerre inequalities of order $2$ for $\Delta p(n)$ and $\Delta\bar{p}(n)$ as showed in Theorem \ref{LIv} and Theorem \ref{LIbarv}, respectively. In fact, the approach given in Section \ref{TIp} still works 
here. The key point in the proof is 
to determine suitable values for \(l\) (resp. $t$) in Lemma \ref{boundp} (resp. Lemma \ref{boundbarp}), choose an appropriate denominator as in \eqref{TIneq vp}, and adjust the terms of the Taylor expansion for the exponential function, as well as the upper and lower bounds for \(\sqrt{w^2-\frac{2\delta\pi^2}{3}}\) and \(\sqrt{w^2+\frac{2\eta\pi^2}{3}}\), where $\delta= 1, 2, 3$ and \(\eta = 1, 2\).

 Since the procedure resembles that of Section \ref{TIp}, we will provide the proof only for Theorem \ref{LIv}, omitting details and presenting 
 the value of \(l\), the chosen denominator as in \eqref{TIneq vp}, the terms of the Taylor expansion of \(e^t\), and the terms of the Taylor expansion of the upper and lower bounds of \(\sqrt{w^2-\frac{2\delta\pi^2}{3}}\) and \(\sqrt{w^2+\frac{2\eta\pi^2}{3}}\), where $\delta= 1, 2, 3$ and \(\eta = 1, 2\).

\begin{thm}\label{LIv}
 For $n \geq 301$, $\Delta p(n)$ satisfies the Laguerre inequality of order \(2\),
  {\small
  \begin{equation}\label{LvIne}
			3\left(\Delta p(n+2)\right)^2-4\left(\Delta p(n+1)\right)\left(\Delta p(n+3)\right)+\left(\Delta p(n)\right)\left(\Delta p(n+4)\right)>0.
	\end{equation}
}
\end{thm}
\begin{thm}\label{LIbarv}
  For $n \geq 50$, $\Delta\bar{p}(n)$ satisfies the Laguerre inequality of order \(2\),
  {\small
\begin{equation}\label{LbarvIne}
3\left(\Delta\bar{p}(n+2)\right)^2-4\left(\Delta\bar{p}(n+1)\right)\left(\Delta\bar{p}(n+3)\right)+\left(\Delta\bar{p}(n)\right)\left(\Delta\bar{p}(n+4)\right)>0.
\end{equation}
}
\end{thm}
\begin{proof} [{Proof of Theorem \ref{LIv}}]
Setting \(l=8\) in Lemma \ref{boundp} yields, for \(n \geq 789\),
\begin{equation*}
			e^{\mu} \frac{\sqrt{12} \beta(\mu) \pi^{2}}{36 \mu^{10}}<p(n)<e^{\mu} \frac{\sqrt{12} \alpha(\mu) \pi^{2}}{36 \mu^{10}},
	\end{equation*}
where
\begin{equation*}
			\alpha(t)=t^{8}-t^{7}+1, \quad ~~\beta(t)=t^{8}-t^{7}-1.
	\end{equation*}
Let
\begin{equation*}
			f(n)\coloneqq e^{\mu} \frac{\sqrt{12} \beta(\mu) \pi^{2}}{36 \mu^{10}},\quad~~
g(n)\coloneqq e^{\mu} \frac{\sqrt{12} \alpha(\mu) \pi^{2}}{36 \mu^{10}}.
	\end{equation*}
Then, when we reach the step \eqref{TIneq vp}, we select $f(n+3)^2$ as the denominator.

Using the first $7$ and $8$ terms of the Taylor expansion of the exponential function, we approximate the upper and lower bounds for $t<0$,
\begin{equation*}
\phi(t)<e^{t}<\Phi(t),
\end{equation*}
where 
\begin{align*}
\Phi(t)&=1+t +\frac{t ^2}{2}+\frac{t ^3}{6}+\frac{t ^4}{24}+\frac{t ^5}{120}+\frac{t ^6}{720}+\frac{t ^7}{5040},\\[9pt]
\phi(t)&=1+t +\frac{t ^2}{2}+\frac{t ^3}{6}+\frac{t ^4}{24}+\frac{t ^5}{120}+\frac{t ^6}{720}+\frac{t ^7}{5040}+\frac{t ^8}{40320}.
\end{align*}

More precise estimates for $x$, $y$, $z$, $r$, and $j$ are required, as in the proof of Theorem \ref{TIv}, where
\begin{align*}
&x=\sqrt{w^{2}-2 \pi^{2}}, \quad~y=\sqrt{w^{2}-\frac{4 \pi^{2}}{3}}, \quad~z=\sqrt{w^{2}-\frac{2 \pi^{2}}{3}},\\[9pt]
&r=\sqrt{w^{2}+\frac{2 \pi^{2}}{3}}, \quad~j=\sqrt{w^{2}+\frac{4 \pi^{2}}{3}}.
\end{align*}

We use the first $5$ and $6$ terms of the Taylor expansion of $x$, $y$, $z$, $r$ and $j$ to approximate their upper and lower bounds. It can be checked that for $n \geq 1$,
\begin{equation*}
	x_1<x<x_2,~~ y_1<y<y_2,~~z_1<z<z_2,~~r_1<r<r_2,~~ j_1<j<j_2,
\end{equation*}
where
\begin{align*}
x_1&=w-\frac{\pi^2}w-\frac{\pi^4}{2w^3}-\frac{\pi^6}{2w^5}-\frac{5\pi^8}{8w^7}-\frac{7\pi^{10}}{8w^9},\\[9pt]
x_2&=w-\frac{\pi^2}w-\frac{\pi^4}{2w^3}-\frac{\pi^6}{2w^5}-\frac{5\pi^8}{8w^7}-\frac{6\pi^{10}}{8w^9},\\[9pt]
y_1&=w-\frac{2\pi^{2}}{3w}-\frac{2\pi^{4}}{9w^{3}}-\frac{4\pi^{6}}{27w^{5}}-\frac{10\pi^{8}}{81w^{7}}-\frac{28\pi^{10}}{243w^{9}},\\[9pt]
y_2&=w-\frac{2\pi^{2}}{3w}-\frac{2\pi^{4}}{9w^{3}}-\frac{4\pi^{6}}{27w^{5}}-\frac{10\pi^{3}}{81w^{7}}-\frac{27\pi^{10}}{243w^{9}},\\[9pt]
z_1&=w-\frac{\pi^2}{3w}-\frac{\pi^4}{18w^3}-\frac{\pi^6}{54w^5}-\frac{5\pi^8}{648w^7}-\frac{7\pi^{16}}{1944w^9},\\[9pt]
z_2&=w-\frac{\pi^2}{3w}-\frac{\pi^4}{18w^3}-\frac{\pi^6}{54w^5}-\frac{5\pi^8}{648w^7}-\frac{6\pi^{16}}{1944w^9},\\[9pt]
r_1&=w+\frac{\pi^2}{3w}-\frac{\pi^4}{18w^3}+\frac{\pi^6}{54w^5}-\frac{5\pi^8}{648w^7},\\[9pt]
r_2&=w+\frac{\pi^2}{3w}-\frac{\pi^4}{18w^3}+\frac{\pi^6}{54w^5}-\frac{5\pi^8}{648w^7}+\frac{7\pi^{19}}{1944w^9},\\[9pt]
j_1&=w+\frac{2\pi^2}{3w}-\frac{2\pi^4}{9w^3}+\frac{4\pi^6}{27w^5}-\frac{10\pi^8}{81w^7},\\[9pt]
j_2&=w+\frac{2\pi^2}{3w}-\frac{2\pi^4}{9w^3}+\frac{4\pi^6}{27w^5}-\frac{10\pi^8}{81w^7}+\frac{28\pi^{10}}{243w^9}.
\end{align*}

Similar to Section \ref{TIp}, we observe $A_1(w)$ is a certain polynomial of degree $68$ with positive coefficient for the first term. For $n\geq 4277$, $A(w)$ is positive. Thus, for $n\geq 4277$, \eqref{LvIne} holds. Direct calculation reveals that for $301\leq n \leq 4276$, $\Delta p(n)$ also satisfies \eqref{LvIne}. The proof of Theorem \ref{LIv} is thus completed.
\end{proof}
The proof of Theorem \ref{LIbarv} closely follows the structure of Theorem \ref{LIv}. Utilizing Lemma \ref{boundbarp} and maintaining the notations introduced in the proof of Theorem \ref{LIv}, we adapt $p(n)$ and $\mu(n)$ to $\bar{p}(n)$ and $\bar{\mu}(n)$, respectively, in line with Section \ref{TIp}. Details are omitted for brevity, and we present the final result.

Applying arguments similar to those in the proof of Theorem \ref{LIv}, we confirm that $A_1(w)$, a polynomial of degree $68$ with a positive leading coefficient, remains positive for $n\geq 2868$. This establishes \eqref{LbarvIne} for $\Delta\bar{p}(n)$ within this range. Further numerical evidence reveals that for $50\leq n \leq 2867$, $\Delta\bar{p}(n)$ also satisfies \eqref{LbarvIne}. Thus, the proof of Theorem \ref{LIbarv} is concluded.
\qed

\section{The determinantal inequalities for $\Delta p(n)$ and $\Delta\bar{p}(n)$}\label{DIp}
In this section, we investigate the determinantal inequalities of order $3$ for $\Delta p(n)$ and $\Delta\bar{p}(n)$, as stated in Theorem \ref{DIv} and Theorem \ref{DIbarv}, respectively. The proof methodology for Theorem \ref{DIbarv} follows the structure established in proving Theorem \ref{DIv}. Thus, we concentrate on a detailed proof for the former, with details analogous to those in Section \ref{TIp}. Here, we present only the value of \(l\), the denominator as in \eqref{TIneq vp}, the terms of the Taylor expansion of \(e^t\), and the terms of the Taylor expansion of the upper and lower bounds of \(\sqrt{w^2-\frac{2\delta\pi^2}{3}}\) and \(\sqrt{w^2+\frac{2\eta\pi^2}{3}}\), where $\delta= 1, 2, 3$ and $\eta = 1, 2$.

\begin{thm}\label{DIv}
  For $n \geq 345$, $\Delta p(n)$ satisfies the determinant inequality,
  \begin{eqnarray}\label{DIneq v}
   \left|\begin{array}{ccc}
		\Delta p(n+2) & \Delta p(n+3) &  \Delta p(n+4)\\
		\Delta p(n+1) & \Delta p(n+2) & \Delta p(n+3)\\
		\Delta p(n) & \Delta p(n+1) & \Delta p(n+2)\\
		\end{array} \right|>0.
	\end{eqnarray}
\end{thm}
\begin{thm}\label{DIbarv}
  For $n \geq 62$, $\Delta\bar{p}(n)$ satisfies the determinant inequality,
  \begin{eqnarray}\label{DIneq barv}
   \left|\begin{array}{ccc}
		\Delta\bar{p}(n+2) & \Delta\bar{p}(n+3) &  \Delta\bar{p}(n+4)\\
		\Delta\bar{p}(n+1) & \Delta\bar{p}(n+2) & \Delta\bar{p}(n+3)\\
		\Delta\bar{p}(n) & \Delta\bar{p}(n+1) & \Delta\bar{p}(n+2)\\
		\end{array} \right|>0.
	\end{eqnarray}
\end{thm}
\begin{proof}[Proof of Theorem \ref{DIv}]
Setting $l=12$ in Lemma \ref{boundp} yields, for $n \geq 2120$,
\begin{equation*}
			e^{\mu} \frac{\sqrt{12} \beta(\mu) \pi^{2}}{36 \mu^{14}}<p(n)<e^{\mu} \frac{\sqrt{12} \alpha(\mu) \pi^{2}}{36 \mu^{14}},
	\end{equation*}
where
\begin{equation*}
			\alpha(t)=t^{12}-t^{11}+1, \quad ~~\beta(t)=t^{12}-t^{11}-1.
	\end{equation*}
Let
\begin{equation*}
			f(n)\coloneqq e^{\mu} \frac{\sqrt{12} \beta(\mu) \pi^{2}}{36 \mu^{14}},\quad~~
g(n)\coloneqq e^{\mu} \frac{\sqrt{12} \alpha(\mu) \pi^{2}}{36 \mu^{14}},
	\end{equation*}
and choose $f(n+1)^3$ as the denominator when we reach the step \eqref{TIneq vp}.

Denote $\Phi(t)$ and $\phi(t)$ as the first $11$ and $12$ terms of the Taylor expansion of $e^t$. Then, for $t<0$, we have $\phi(t)<e^{t}<\Phi(t)$.
We employ the first $7$ and $8$ terms of the Taylor expansion of \(\sqrt{w^2-\frac{2\delta\pi^2}{3}}\) and \(\sqrt{w^2+\frac{2\eta\pi^2}{3}}\) to approximate their upper and lower bounds, where $\delta= 1, 2, 3$ and \(\eta = 1, 2\), as in the proof of Theorem \ref{LIv}.

With the analogous arguments from Section \ref{LIp}, we determine that $A_1(w)$ is a certain polynomial of degree $150$ with a positive leading coefficient. For $n\geq 45284$, $A(w)$ is positive, leading to 
\eqref{DIneq v} for $\Delta p(n)$. The case for $345\leq n \leq 45283$ can be verified with Mathematica. Hence, the proof of Theorem \ref{DIv} is completed.
\end{proof}
The proof of Theorem \ref{DIbarv} closely follows the structure of Theorem \ref{DIv}. Utilizing Lemma \ref{boundbarp} and maintaining notations from the proof of Theorem \ref{DIv}, we adapt $p(n)$ and $\mu(n)$ to $\bar{p}(n)$ and $\bar{\mu}(n)$, respectively, in line with Section \ref{TIp}. Details are omitted for brevity, and we present the final result.

Applying arguments akin to those in the proof of Theorem \ref{DIv}, we observe that $A_1(w)$ is a polynomial of degree $150$ with a positive leading coefficient. Verification confirms $A(w)>0$ for $n\geq 22275$, establishing \eqref{DIneq barv} for $\Delta\bar{p}(n)$ within this range. Further numerical evidence reveals that for $62\leq n \leq 22274$, $\Delta\bar{p}(n)$ also satisfies \eqref{DIneq barv}. Thus, the proof of Theorem \ref{DIbarv} is concluded.
\qed

\section{Open problems}\label{Conj}
In this section, we conjecture thresholds for Laguerre inequalities of order $m$ and the positivity of $m$-th order determinants. Moreover, we conjecture thresholds for the third-order Tur\'an inequalities and the positivity of invariant $I$. These conjectures apply to the $k$-th order differences of $p(n)$ and $\bar{p}(n)$, where $1 \leq m \leq 11$ and $1 \leq k \leq 5$.
\begin{conj}
    For $1 \leq m \leq 11$ and $1 \leq k \leq 5$$:$
    
    \begin{enumerate}
        \item  $\Delta^k p(n)$ $(\text{resp.}~\Delta^k \bar{p}(n))$ satisfies the Laguerre inequality of order $m$ when $n \geq L_p(k, m)$ $(\text{resp.}~n \geq L_{\bar{p}}(k, m))$.
        \item The $m$-th order determinants of $(\text{resp.}~\Delta^k \bar{p}(n))$ are positive for $n \geq D_p(k, m)$ $(\text{resp.}~n \geq D_{\bar{p}}(k, m))$, except when both $k\geq3$ and $m$ are odd.
    \end{enumerate}
\end{conj}

The values of $L_p(k, m)$, $L_{\bar{p}}(k, m)$, $D_p(k, m)$ and $D_{\bar{p}}(k, m)$ are specified in Tables \ref{L_p}, \ref{L_P}, \ref{D_p} and \ref{D_P}, respectively.
    The symbol $\times$ indicates that there does not exist a value satisfying corresponding conjecture.
    
\begin{conj}\label{conj}
    For $1 \leq k \leq 5$$:$
    
    \begin{enumerate}
        \item  $\Delta^k p(n)$ $(\text{resp.}~\Delta^k \bar{p}(n))$ satisfies the third-order Tur\'an inequality when $n \geq T_p(k)$ $(\text{resp.}~ n \geq T_{\bar{p}}(k))$.
        \item  $\Delta^k p(n)$ $(\text{resp.}~\Delta^k \bar{p}(n))$ satisfies $I>0$ when $n \geq I_p(k)$ $(\text{resp.}~n \geq I_{\bar{p}}(k))$.
    \end{enumerate}
\end{conj}
The values of $T_p(k)$, $T_{\bar{p}}(k)$, $I_p(k)$ and $I_{\bar{p}}(k)$ are specified in Table \ref{T_p}.

\begin{table}[htbp]
\centering
\resizebox{\textwidth}{!}{
\begin{tabular}{|c|*{11}{>{\centering\arraybackslash}c|}}
  \hline
  \begin{minipage}{1.8cm}
    \begin{tikzpicture}[baseline=(current bounding box.center), trim left=0.2cm]
      \draw (0,1.5) -- (0.6,0);
      \draw (0,1.5) -- (2.2,0.8);
      \node at (0.2,0.5) {$k$};
      \node at (1.2,0.6) {$L_p(k,m)$};
      \node at (1.8,1.2) {$m$};
    \end{tikzpicture}
  \end{minipage}
  & 1 & 2 & 3 & 4 & 5 & 6 & 7 & 8 & 9 & 10 & 11 \\
\hline
1 & 70 & 301 & 738 & 1413 & 2346 & 3557 & 5062 & 6873 & 9006 & 11467 & 14270 \\
2 & 138 & 451 & 986 & 1767 & 2816 & 4151 & 5786 & 7733 & 10006 & 12613 & 15566 \\
3 & 234 & 637 & 1272 & 1767 & 3334 & 4795 & 6562 & 8649 & 11064 & 13819 & 16922 \\
4 & 362 & 859 & 1602 & 2609 & 2346 & 5491 & 7394 & 9619 & 12180 & 15083 & 18340 \\
5 & 522 & 1121 & 1974 & 3101 & 4518 & 6241 & 8280 & 10649 & 13356 & 16411 & 19822 \\
\hline
\end{tabular}
}
\caption{The values of $L_p(k,m)$.}
\label{L_p}
\end{table}

\begin{table}[htbp]
\centering
\resizebox{\textwidth}{!}{
\begin{tabular}{|c|*{11}{>{\centering\arraybackslash}c|}}
  \hline
  \begin{minipage}{1.8cm}
    \begin{tikzpicture}[baseline=(current bounding box.center), trim left=0.2cm]
      \draw (0,1.5) -- (0.6,0);
      \draw (0,1.5) -- (2.2,0.8);
      \node at (0.2,0.5) {$k$};
      \node at (1.2,0.6) {$L_{\bar{p}}(k,m)$};
      \node at (1.8,1.2) {$m$};
    \end{tikzpicture}
  \end{minipage}
 & 1 & 2 & 3 & 4 & 5 & 6 & 7 & 8 & 9 & 10 & 11 \\
\hline
1 & 7 & 50 & 142 & 294 & 509 & 799 & 1167 & 1616 & 2146 & 2778 & 3497 \\
2 & 21 & 89 & 208 & 390 & 641 & 967 & 1371 & 1859 & 2440 & 3111 & 3785 \\
3 & 45 & 131 & 277 & 489 & 773 & 1126 & 1575 & 2105 & 2722 & 3435 & 4241 \\
4 & 69 & 179 & 352 & 588 & 908 & 1300 & 1779 & 2345 & 3001 & 3753 & 4601 \\
5 & 98 & 232 & 429 & 698 & 1045 & 1473 & 1985 & 2587 & 3282 & 4073 & 4963 \\
\hline
\end{tabular}
}
\caption{The values of $L_{\bar{p}}(k,m)$.}
\label{L_P}
\end{table}

\begin{table}[htbp]
\centering
\resizebox{\textwidth}{!}{
\begin{tabular}{|c|*{11}{>{\centering\arraybackslash}c|}}
  \hline
  \begin{minipage}{1.8cm}
    \begin{tikzpicture}[baseline=(current bounding box.center), trim left=0.2cm]
      \draw (0,1.5) -- (0.6,0);
      \draw (0,1.5) -- (2.2,0.8);
      \node at (0.2,0.5) {$k$};
      \node at (1.2,0.6) {$D_p(k,m)$};
      \node at (1.8,1.2) {$m$};
    \end{tikzpicture}
  \end{minipage}
  & 1 & 2 & 3 & 4 & 5 & 6 & 7 & 8 & 9 & 10 & 11 \\
\hline
1 & 1 & 69 & 345 & 879 & 1709 & 2857 & 4347 & 6197 & 8419 & 11029 & 14039 \\
2 & 7  & 137 & 503 & 1145 & 2091 & 3369 & 4995 & 6987 & 9359 & 12123 & 15293 \\
3 & $\times$  & 233 & $\times$  & 1451 & $\times$ &  3929 & $\times$ &  7831 & $\times$ &  13275 & $\times$ \\
4 & 67  & 361 & 929 & 1797 & 2997 & 4537 & 6443 & 8729 & 11405 & 14485 & 17979 \\
5 & $\times$  & 521 & $\times$  & 2189 & $\times$ &  5197 & $\times$ &  9683 & $\times$ &  15755 & $\times$ \\
\hline
\end{tabular}
}
\caption{The values of $D_p(k,m)$.}
\label{D_p}
\end{table}

\begin{table}[htbp]
\centering
\resizebox{\textwidth}{!}{
\begin{tabular}{|c|*{11}{>{\centering\arraybackslash}c|}}
  \hline
  \begin{minipage}{1.8cm}
    \begin{tikzpicture}[baseline=(current bounding box.center), trim left=0.2cm]
      \draw (0,1.5) -- (0.6,0);
      \draw (0,1.5) -- (2.2,0.8);
      \node at (0.2,0.5) {$k$};
      \node at (1.2,0.6) {$D_{\bar{p}}(k,m)$};
      \node at (1.8,1.2) {$m$};
    \end{tikzpicture}
  \end{minipage}
   & 1 & 2 & 3 & 4 & 5 & 6 & 7 & 8 & 9 & 10 & 11 \\
\hline
1 & 1 & 6 & 62 & 185 & 389 & 674 & 1055 & 1535 & 2120 & 2813 & 3620 \\
2 & 1  & 20 & 104 & 257 & 494 & 821 & 1241 & 1766 & 2396 & 3134 & 3992 \\
3 & $\times$  & 44 & $\times$ & 335 & $\times$ & 965 & $\times$ & 1991 & $\times$ & 3449 & $\times$ \\
4 & 10  & 68 & 200 & 416 & 716 & 1115 & 1613 & 2216 & 2930 & 3758 & 4706 \\
5 & $\times$  & 97 & $\times$ & 496 & $\times$ & 1264 & $\times$ & 2440 & $\times$ & 4066 & $\times$ \\
\hline
\end{tabular}
}
\caption{The values of $D_{\bar{p}}(k,m)$.}
\label{D_P}
\end{table}

\begin{table}[htbp]
  \centering
   \renewcommand{\arraystretch}{1}
    \setlength{\tabcolsep}{3pt}
\begin{tabular}{|c|*{5}{>{\centering\arraybackslash}c|}}
\hline
$k$ & 1 & 2 & 3 & 4 & 5 \\
\hline
$T_p(k)$ & 174 & 284 & 424 & 600 & 810 \\
$T_{\bar{p}}(k)$ & 33 & 57 & 87 & 118 & 173 \\
$I_p(k)$ & 329 & 483 & 675 & 903 & 1171 \\
$I_{\bar{p}}(k)$ & 64 & 98 & 143 & 194 & 255 \\
\hline
\end{tabular}
\caption{The values of Conjecture \ref{conj}.}
\label{T_p}
\end{table}

\noindent {\bf Conclusion.} The proof strategy for all inequalities in this paper relies on the key Lemma \ref{boundp} and Lemma \ref{boundbarp}. The process involved determining an appropriate \(l\) (resp. $t$) in Lemma \ref{boundp} (resp. Lemma \ref{boundbarp}), reasonably selecting the denominator as in \eqref{TIneq vp} to ensure the negativity of the power part of $e^t$, truncating the terms of the Taylor expansion for $e^t$, \(\sqrt{w^2-\frac{2\delta\pi^2}{3}}\) and \(\sqrt{w^2+\frac{2\eta\pi^2}{3}}\) (resp. \(\sqrt{w^2-\delta\pi^2}\) and \(\sqrt{w^2+\eta\pi^2}\)), where $\delta= 1, 2, \ldots, m$ and \(\eta = 1, 2, \ldots, m-1\), at a suitable point depending on the Laguerre inequality of order $m-1$ or the positivity of $m$-th order determinants under consideration.

{\centering\section*{Acknowledgments}}
 I am grateful to Prof. Qing-Hu Hou for his dedicated guidance and to the referees for their valuable feedback. This research was financially supported by the National Natural Science Foundation of China (Grant 11921001).


\begin{thebibliography}{99}
\bibitem{Almkvist} G. Almkvist, Exact Asymptotic Formulas for the Coefficients of Nonmodular Functions, J. Number Theory. 38 (1991), 145--160.

\bibitem{Almkvist2} G. Almkvist, On the differences of the partition function, Acta Arith. 61 (1992), 173--181.

\bibitem{Banerjee} K. Banerjee, Inequalities for the modified bessel function of first kind of non-negative order, J. Math. Anal. Appl. 524 (1) (2023), 127082.

\bibitem{Canfield} R. Canfield, S. Corteel and P. Hitczenko, Random partitions with non-negative r-th differences, Adv. Appl. Math. 27 (2001), 298--317.

\bibitem{Cardon} D.A. Cardon and A. Rich, Tur\'an inequalities and subtraction-free expressions, JIPAM. J. Inequal. Pure Appl. Math. 9 (4) (2008), Artical 91. 11 pp.

\bibitem{Chen3} W.Y.C. Chen, The spt-Function of Andrews, Sueveys in combinatorics, London Math. Soc. Lecture Note Ser., 440, Cambridge Univ. Press, Cambridge, 2017, 141--203.

\bibitem{Chen-Jia-Wang-2017} W.Y.C. Chen, D.X.Q. Jia and L.X.W. Wang, Higher order Tur\'an inequalities for the partition function, Trans. Amer. Math. Soc.  372 (3)
(2019), 2143--2165.

\bibitem{Chen-Wang-Xie} W.Y.C. Chen, L.X.W. Wang and G.Y.B. Xie, Finite differences of the logarithm of the partition function, Math. Comp. 85 (2016), 825--847.


\bibitem{Craig} W. Craig and A. Pun, A note on the higher order Tur\'{a}n inequalities for $k$-regular partitions, Res. Number Theory. 7 (5) (2021), 1--7.

\bibitem{Craven} T. Craven and G. Csordas, Jensen polynomials and the Tur\'{a}n and Laguerre inequalities, Pacific J. Math. 136 (2) (1989), 241--260. 

\bibitem{Csordas5} T. Craven and G. Csordas, Iterated Laguerre and Tur\'{a}n inequalities, JIPAM. J. Inequal. Pure Appl. Math. 3 (3) (2002), Artical 39. 14 pp. 

\bibitem{Csordas3} G. Csordas, Complex zero decreasing sequences and the Riemann Hypothesis II, Analysis and applications-ISAAC 2001 (Berlin), Int. Soc. Anal. Appl. Comput., vol. 10, Kluwer Acad. Publ., Dordrecht, 2003, 121--134.

\bibitem{Csordas} G. Csordas, T.S. Norfolk and R.S. Varga, The Riemann hypothesis and the Tur\'an inequalities, Trans. Amer. Math. Soc. 296 (2) (1986), 521--541. 

\bibitem{Csordas2} G. Csordas and R.S. Varga, Necessary and sufficient conditions and the Riemann hypothesis, Adv. Appl. Math. 11 (3) (1990), 328--357. 


\bibitem{Desalvo} S. DeSalvo and I. Pak, Log-concavity of the partition function, Ramanujan J. 38 (1) (2015), 61--73.

\bibitem{Engel} B. Engel, Log-concavity of the overpartition function [J], Ramanujan J. 43 (2) (2017), 229--241.

\bibitem{K.Dilcher} K. Dilcher and K.B. Stolarsky, On a class of nonlinear differential operators acting on polynomials, J. Math. Anal. Appl. 170 (2) (1992), 382--400. 

\bibitem{Dimitrov} D.K. Dimitrov, Higher order Tur\'{a}n inequalities, Proc. Amer. Math. Soc. 126 (7) (1998), 2033--2037.

\bibitem{Dimitrov2} D.K. Dimitrov and F.R. Lucas, Higher order Tur\'an inequalities for the Riemann $\xi$-function, Proc. Amer. Math. Soc. 139 (3) (2011), 1013--1022. 

\bibitem{DJ} J.J.W. Dong and K.Q. Ji, Higher order Tur\'an inequalities for the distinct partition function, 2023, arXiv:2303.05243.

\bibitem{Dong} J.J.W. Dong, K.Q. Ji and D.X.Q. Jia, Tur\'an inequalities for the broken $k$-diamond partition function, Ramanujan J. 62 (2) (2023), 593--615.

\bibitem{dw} L.M. Dou and L.X.W. Wang, Higher order Laguerre inequalities for the partition function, Discret. Math. 346 (6) (2023), 113366.


\bibitem{W.H.Foster} W.H. Foster and I. Krasikov, Inequalities for real-root polynomials and entire functions, Adv. Appl. Math. 29 (1) (2002), 102--114. 

\bibitem{Gomez} K. Gomez, J. Males and L. Rolen, The second shifted difference of partitions and its applications, Bull. Aust. Math. Soc. 107 (2023), 66--78.

\bibitem{Good} I.J. Good, The difference of the partition function, Problem 6137, Am. Math. Mon. 84 (1997), 141.

\bibitem{Zagier}  M. Griffin, K. Ono, L. Rolen and D. Zagier, Jensen polynomials for the Riemann zeta function and other sequences, Proc. Natl. Acad. Sci. 116 (23) (2019), 11103--11110. 
    
\bibitem{Gupta} H. Gupta, Finite differences of the partition function, Math. Comp. 32 (1978), 1241--1243.
    
\bibitem{hz} Q. Hou and Z. Zhang,  $r$-log-concavity of partition functions, Ramanujan J. 48 (1) (2019), 117--129.

\bibitem{Jensen} J.L.W.V. Jensen, Recherches sur la th\'eorie des \'equations, Acta Math. 36 (1913), 181--195. 

\bibitem{Jia} D.X.Q. Jia, Inequalities for the broken $k$-diamond partition function, J. Number Theory. 249 (2023), 314--347.

\bibitem{Jia-Wang-2018} D.X.Q. Jia and L.X.W. Wang, Determinantal inequalities of partition function, Proc. Royal Soc. Edinb. A. 150 (3) (2020), 1451--1466.

\bibitem{Knessl00} C. Knessl, Asymptotic behavior of high-order differences of the plane partition function, Discret. Math. 126 (1994), 179--193.


\bibitem{Knessl0} C. Knessl and J.B. Keller, Partition asymptotics from recursion equations, SIAM J. Appl. Math. 50 (1990), 323--338.

\bibitem{Knessl} C. Knessl and J.B. Keller, Asymptotic behavior of high-order differences of the partition function, Commun. Pure Appl. Math. 44 (1991), 1033--1045.


\bibitem{Laguerre} E. Laguerre, Oeuvres, vol.1,  (Paris: Gaauthier-Villars, 1898). 


\bibitem{larson} H. Larson and I. Wagner, Hyperbolicity of the partition Jensen polynomials, Res. Number Theory. 5 (19) (2019), 1--12. 

\bibitem{Lehmer1} D.H. Lehmer, On the series for the partition function, Trans. Amer. Math. Soc., 43 (2) (1938), 271--295. 


\bibitem{Lehmer2} D.H. Lehmer, On the Remainders and Convergence of the Series for the Partition Function [J], Trans. Amer. Math. Soc. 46 (1939), 362--373. 


\bibitem{Levin} B.J. Levin, Distribution of zeros of entire functions, Gosudarstv. Izdat. Tehn.-Teor. Lit., Moscow, 1956, 632 pp.



\bibitem{Liu} E.Y.S. Liu and H.W.J. Zhang, Inequalities for the overpartition function, Ramanujan J. 54 (3) (2021), 485--509.

\bibitem{Merca} M. Merca and J. Katriel, A general method for proving the non-trivial linear homogeneous partition inequalities, Ramanujan J. 51 (2020), 245--266.


\bibitem{Nicolas} J.L. Nicolas, Sur les entiers N pour lesquels il y a beaucoup de groupes ab\'eliens $d^{\prime}$ordre N, Ann. Inst. Fourier. 28 (4) (1978), 1--16.

\bibitem{Odlyzko} A.M. Odlyzko, Differences of the partition function, Acta Arith. 49 (1988), 237--254.


\bibitem{Patrick1} M.L. Patrick, Some inequalities concerning Jacobi polynomials, SIAM J. Math. Anal. 2 (2) (1971), 213--220. 

\bibitem{Patrick} M.L. Patrick, Extensions of inequalities of the Laguerre and Tur\'an type, Pacific J. Math. 44 (2) (1973), 675--682. 

\bibitem{Polya} G. P\'olya, \"{U}ber die algebraisch-funktionentheoretischen Untersuchungen von J.L.W.V. Jensen, Kgl. Danske Vid. Sel. Math.-Fys. Medd. 7 (1927), 3--33.


\bibitem{Rahman} Q.I. Rahman and G. Schmeisser, Analytic theroy of polynomials, Oxford University Press, Oxford, 2002. xiv+742 pp. 

\bibitem{Skovgaard} H. Skovgaard, On inequalities of the Tur\'an type, Math. Scand. 2 (1954), 65--73.

\bibitem{Wagner} I. Wagner, On a new class of Laguerre-P\'olya type functions with applications in number theory, Pacific J. Math. 320 (1) (2022), 177--192.

\bibitem{wang1} L.X.W. Wang, G.Y.B. Xie and A.Q. Zhang, Finite difference of the overpartition function, Adv. Appl. Math. 92 (2018), 51--72.

\bibitem{wang2} L.X.W. Wang and E.Y.Y. Yang, Laguerre inequalities for discrete sequences, Adv. Appl. Math. 139 (2022), 102357.


\bibitem{wang3} L.X.W. Wang and N.N.Y. Yang, Positivity of the determinants of the partition function and the overpartition function, Math. Comput. 341 (92) (2023), 1383--1402.

\bibitem{wang4} L.X.W. Wang and E.Y.Y. Yang, Laguerre inequality and determinantal inequality for the distinct partition function, submitted.
    
        
\bibitem{Yang} E.Y.Y. Yang, Finite differences of the distinct partition function and the broken $k$-diamond partition function, submitted.

\bibitem{Yang1} E.Y.Y. Yang, Laguerre inequality and determinantal inequality for the broken $k$-diamond partition function, submitted.
    
\bibitem{Zuckerman} H.S. Zuckerman, On the coefficients of certain modular forms belonging to subgroups of the modular group. Trans. Am. Math. Soc. 45 (2) (1939), 298--321. 
\end{thebibliography}
\end{document}